\documentclass[11pt]{amsart}
\usepackage[centertags]{amsmath}
\usepackage{amsfonts,amssymb}
\usepackage{graphicx}
\usepackage{enumerate}

  \newtheorem{theorem}{Theorem}
  \newtheorem{corollary}{Corollary}

  \newtheorem{lemma}{Lemma}
  
  \DeclareMathOperator{\re}{Re}

\DeclareMathOperator{\res}{Res}

\begin{document}

\title{A sieve problem and its application}

\author{Andreas Weingartner} 
\address{Department of Mathematics, Southern Utah University, 351 West University Boulevard, Cedar City, Utah 84720, USA} 
\email{weingartner@suu.edu} 
\begin{abstract}
Let $\theta$ be an arithmetic function and let $\mathcal{B}$ be the set of positive integers $n=p_1^{\alpha_1} \cdots p_k^{\alpha_k}$, which
satisfy $p_{j+1} \le \theta ( p_1^{\alpha_1}\cdots p_{j}^{\alpha_{j}})$  for $0\le j < k$. We show that $\mathcal{B}$ has a natural density, provide a criterion to determine whether this density is positive, and give various estimates for the counting function of $\mathcal{B}$. 
When $\theta(n)/n$ is non-decreasing, the set $\mathcal{B}$ coincides with the set of integers $n$ 
whose divisors $1=d_1< d_2 < \ldots <d_{\tau(n)}=n$ satisfy 
$d_{j+1} \le \theta( d_j  )$ for $1\le j <\tau(n)$.
\end{abstract} 
\maketitle

\section{Introduction}

Let $\theta$ be an arithmetic function, $\theta: \mathbb{N}\to \mathbb{R}\cup \{\infty\}$.  We write $\mathcal{B}$ (or $\mathcal{B}_\theta$) to denote the set of positive integers containing $n=1$ and all those $n \ge 2$ with prime factorization  $n=p_1^{\alpha_1} \cdots p_k^{\alpha_k}$, \
$p_1< p_2< \ldots < p_k$, which satisfy 
\begin{equation}\label{Bdef}
p_{j+1} \le \theta (p_1^{\alpha_1}\cdots p_{j}^{\alpha_{j}} ) \qquad (0\le j < k),
\end{equation}
where $p_1^{\alpha_1}\cdots p_{j}^{\alpha_{j}}$ 
is understood to be $1$ when $j=0$. 
Let $B(x)$ (or $B_\theta(x)$) be the number of positive integers $n\le x$ in $\mathcal{B}$.
The following list shows some examples of $\theta$ and its corresponding set $\mathcal{B}$.
\medskip

\begin{itemize}
 \item   If $\theta(n)=2n$,   then $\mathcal{B}$  is the set of integers with $2$-dense divisors, i.e. integers $n$ which have a divisor
 in every interval $[y,2y]$ for $1\le y \le n$ (see \cite{Saias1,Ten86,PDD}).
\item    If $\theta(n)=\sigma(n)+1$,  where $\sigma(n)$ is the sum-of-divisors function, 
then $\mathcal{B}$  is the set of practical numbers, i.e. integers $n$ such that every 
$1\le m\le n$ can be written as a sum of distinct positive divisors of $n$ (see \cite{Saias1,Ten86,PDD} and the references therein).
 \item   If $\theta(n) = n+1$,  then $\mathcal{B}$  is the set of even $\varphi$-practical numbers, i.e. even integers $n$ such that 
 the polynomial $X^n-1$ has a divisor in $\mathbb{Z}[X]$ of every degree from $1$ to $n$ (see \cite{PTW,Thom,Thom2}).
\end{itemize}

Building on earlier work by Tenenbaum \cite{Ten86} and Saias \cite{Saias1}, we found 
\cite[Theorem 1.2]{PDD} that 
$$B(x)\sim \frac{c_\theta x}{\log x} \qquad (x \to \infty),$$
for some positive constant $c_\theta$,
provided $\theta(n)$ satisfies
$$\max(2,n)\le \theta(n) \ll \frac{n \log 2n}{(\log\log 3n)^{1+\varepsilon} }\quad (n\ge 1),$$
for some $\varepsilon>0$. This applies to each of the three examples listed above.

In this note, our goal is to investigate the set $\mathcal{B}$ in general,
without any restrictions on $\theta$. 
We will show that $\mathcal{B}$ always has a natural density (Theorem \ref{thmB}) and
provide a criterion to determine whether this natural density is positive or zero (Theorem \ref{cor}).
We give estimates for $B(x)$ with explicit error terms, first
without any assumptions on $\theta$ (Theorem \ref{thmB}), and then under certain conditions 
on the size of $\theta(n)$ (Corollary \ref{corB} and Theorem \ref{thm4}).

As an application, we consider the following set, related to the distribution of divisors.
Let $\mathcal{D}$ be the set of positive integers containing $n=1$ and all those $n \ge 2$, whose divisors  $1=d_1< d_2 < \ldots <d_{\tau(n)}=n$ satisfy 
\begin{equation}\label{Ddef}
d_{j+1} \le \theta( d_j  ) \qquad (1\le j <\tau(n)).
\end{equation}
We write $D(x)$ for the number of positive integers $n\le x$ in $\mathcal{D}$.
Theorem \ref{DB} shows that $\mathcal{B}=\mathcal{D}$ provided $\theta(n)/n$ is non-decreasing, so that
all results concerning $B(x)$ also apply to $D(x)$ under this assumption.

\section{Statement of results}

Let $P^+(n)$ (resp. $P^-(n)$) denote the largest (resp. smallest) prime factor of $n\ge 2$ and put $ P^+(1)=1$, $ P^-(1)=\infty$.

Note that replacing $\theta(n)$ by $\max(\theta(n),P^+(n))$ in \eqref{Bdef}  
leaves the set $\mathcal{B}$ unchanged, because of the assumption $p_1< p_2< \ldots < p_k$.
Moreover, if $\theta(1)<2$ then $\mathcal{B}=\{1\}$.
Thus, we may assume from now on, without any loss of generality, that
\begin{equation}\label{theta}
\theta: \mathbb{N}\to \mathbb{R}\cup \{\infty\}, \quad \theta(1)\ge 2, \quad  \theta(n)\ge P^+(n) \quad (n\ge 2).
\end{equation}

Let $\chi(n)$ be the characteristic function of the set $\mathcal{B}$.
We shall see in  Lemma \ref{lemL} that the series 
\begin{equation*}
L=\sum_{n\ge 1} \frac{\chi(n)}{n} \prod_{p\le \theta(n)} \left(1-\frac{1}{p}\right)
\end{equation*}
converges to a value $0\le L \le 1$.
Theorem \ref{thmB} shows that, for every choice of $\theta$, the set $\mathcal{B}$ has a natural density,
which is given by $1-L$.

\begin{theorem}\label{thmB}
Let $\theta$ satisfy \eqref{theta}.
We have
\begin{equation}\label{thmB2}
B(x) = (1-L)x + o(x).
\end{equation}
More precisely, 
\begin{equation}\label{thmB1}
B(x)=(1-L)x +O\left(1 + x \sum_{n\ge 1} \frac{\chi(n)}{n \log \theta(n)} \exp\left(-\frac{\max(0,\log (x/n))}{3\log \theta(n)}\right)\right).
\end{equation}
\end{theorem}

Theorem \ref{cor} provides a simple criterion to determine whether $L=1$ and $B(x)=o(x)$, or $L<1$ and $B(x)\sim (1-L)x$.

\pagebreak

\begin{theorem}\label{cor}
Let $\theta$ satisfy \eqref{theta}.
\begin{enumerate}
\item[(i)]
Assume that $\theta(n) \le f(n)$ where $n \log f(n)$ is increasing.
If 
\begin{equation*}
 \sum_{n\ge 1} \frac{1}{n \log f(n)} 
\end{equation*}
diverges, then $L=1$ and $B(x)=o(x)$.
\item[(ii)]
Assume that $\theta(n) \ge f(n)\ge 2$, $f(n)\gg P^+(n)$, and that
for every $t\ge 1$ there exists an $r\in \mathbb{N}$, such that $f(2^r n) \ge t f(n)$ for all $n\ge 1$. If 
$$
\sum_{n\ge 1} \frac{1}{n \log f(n)}
$$
converges, then $L<1$ and $B(x)\sim(1-L)x$.
\item[(iii)] $B(x) \sim x $ $ \Leftrightarrow $ $L=0$  $ \Leftrightarrow $  $\theta(n)=\infty $ for all $ n\ge 1$  $ \Leftrightarrow  $ $B(x)=[x]$.
\end{enumerate}
\end{theorem}

Note that the three examples of $\theta$ listed in the introduction satisfy $L=1$ and $B(x)=o(x)$.  
For an instance where $L<1$, consider $\theta(n)=2^n$, for which 
$L=0.7734...$ by numerical computation. 
(For $n\ge 30$ we used estimates for the Euler product with effective error bounds due to Rosser and Schoenfeld \cite[Theorem 7]{RS}.)
Hence \eqref{thmB2} implies that $B(x)=cx(1+o(1))$ with $c=1-L=0.2265...$, while
\eqref{thmB1} shows that we have $B(x)=cx(1+O(1/\log x))$, when $\theta(n)=2^n$.

Corollary \ref{corB} generalizes this example to $\log\theta(n) \asymp n^a$, where $a$ is a positive constant.
We also consider the case $\log\theta(n) \asymp (\log 2n)^a$, where $a>1$ is constant. 
(The notation $f(n) \asymp g(n)$ means that $f(n)\ll g(n)$ (i.e. $f(n)=O(g(n))$) and $g(n)\ll f(n)$.)

\begin{corollary}\label{corB}
Let $\theta$ satisfy \eqref{theta}.
\begin{enumerate}
\item[(i)]
If $\log\theta(n) \asymp n^a$ for some constant $a>0$, then $0<L<1$ and 
$$B(x) = (1-L) x\left(1 + O\left(\frac{1}{\log x}\right)\right).$$
\item[(ii)]
If $\log\theta(n) \asymp (\log 2n)^a$ for some constant $a>1$, then $0<L<1$ and
$$B(x) = (1-L) x\left(1 + O\left(\frac{1}{(\log x)^{1-1/a}}\right)\right).$$
\end{enumerate}
\end{corollary}

The error terms in Corollary \ref{corB} are easily derived from \eqref{thmB1}, using the trivial bound $\chi(n)\le 1$.
The claim that $0<L<1$ follows directly from Theorem \ref{cor}, with (i) $f(n)=\max(2,\exp(c n^a))$ and
(ii) $f(n)=\max(2,\exp(c (\log 2n)^a))$, where $c>0$ is a suitable constant.
Other examples of $\theta$, for which $L<1$, can be dealt with similarly.

When $B(x)=o(x)$, we need a different strategy for obtaining an asymptotic formula for $B(x)$,
since the estimate \eqref{thmB1} provides only an upper bound for $B(x)$ whenever $L=1$. 
We will focus on the case $\theta(n) \asymp n^a$, where $a\ge 1$ is constant. 
Theorem \ref{cor} shows that $L=1$ and $B(x)=o(x)$.
Theorem \ref{thm4} generalizes \cite[Theorem 1.2]{PDD}, where the case $a=1$ is established with $\lambda_1=1$.

\begin{theorem}\label{thm4}
Let $\theta$ satisfy \eqref{theta} and assume $\theta(n)\asymp n^a$ for some constant $a\ge 1$. 
Then there are constants $c_\theta>0$ and $\lambda_a\in (0,1]$, such that
\begin{equation}\label{Basymp}
B(x) =\frac{ c_\theta x}{(\log x)^{\lambda_a}}\left(1+ O\left(\frac{1}{\log x}\right)\right).
\end{equation}
Here $s=-\lambda_a$ is the unique solution in the interval $[-1,0)$ of the equation
\begin{equation}\label{lambdaeq}
0= s + \frac{e^{-\gamma}}{a (a+1)^s} 
+s \int_1^\infty \bigl( \omega(t) - e^{-\gamma}\bigr)   \frac{dt}{(at+1)^{s+1}}, 
\end{equation}
where $\omega(t)$ denotes Buchstab's function and $\gamma$ is Euler's constant. For $a\ge 1$, 
\begin{equation}\label{LB}
  \lambda_a>\frac{e^{-\gamma}}{a}\left(1+\frac{e^{-\gamma}\log (a+1)}{a}-\frac{0.16}{a}\right)
\end{equation}
and
\begin{equation}\label{UB}
\lambda_a< \frac{e^{-\gamma}}{a}\left(1+\frac{e^{-\gamma}\log (a+1)}{a}+\frac{\log^2(a+1)}{a^2}\right). 
\end{equation}
\end{theorem}

Figure \ref{figure1} and Table \ref{table1} show several values of $\lambda_a$, obtained by solving equation \eqref{lambdaeq} numerically.

As in \cite[Theorem 1.2]{PDD} with $a=1$, one can consider a less restrictive condition on $\theta$, such as 
$n^a (\log n)^{-b} \ll \theta(n) \ll n^a (\log n)^{b}$, where $0\le b<1$, and establish the estimate \eqref{Basymp}
with the relative error term $O(1/\log x)$ replaced by $O(1/(\log x)^{1-b})$. However, we will not pursue this here.

 \begin{figure}[h]
\begin{center}
\includegraphics[height=7cm,width=12cm]{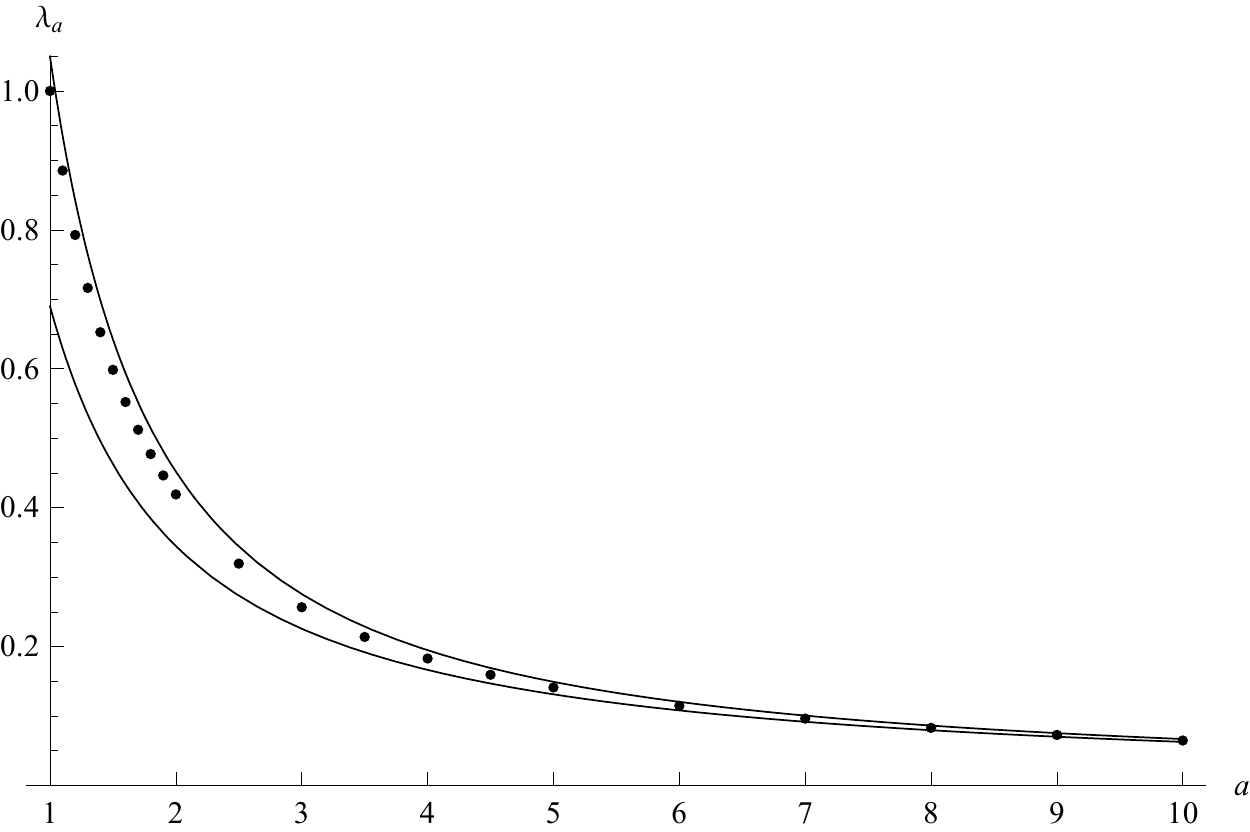}
\caption{Values of $\lambda_a$ together with the bounds \eqref{LB} and \eqref{UB}.}
\label{figure1}
\end{center}
\end{figure}

\begin{table}[h]
  \begin{tabular}{ | l | l | }
    \hline
    $a$ & $\lambda_a$ \\ \hline
    1 & 1  \\ \hline
    1.1 &  0.8854... \\ \hline
    1.2 & 0.7927... \\ \hline
    1.3 & 0.7164... \\ \hline
    1.4 & 0.6526... \\ \hline
    1.5 & 0.5985...   \\ \hline
    1.6 & 0.5522...   \\ \hline
    1.7 & 0.5122...   \\ \hline
    1.8 & 0.4772...   \\ \hline
    1.9 & 0.4464...   \\ \hline
    2 & 0.4191...    \\ \hline
      \end{tabular}
      \qquad
     \begin{tabular}{ | l | l | }
    \hline 
     $a$ & $\lambda_a$ \\ \hline
    2.5 & 0.3195... \\ \hline
    3 & 0.2567...  \\ \hline
    3.5 & 0.2139...  \\ \hline
    4 & 0.1829...  \\ \hline
    4.5 & 0.1595...  \\ \hline
    5 & 0.1412... \\ \hline
    6 & 0.1147... \\ \hline
    7 & 0.09639... \\ \hline
    8 & 0.08301... \\ \hline
    9 & 0.07283... \\ \hline
    10 & 0.06484... \\ \hline
  \end{tabular}
  \bigskip
  \caption{Truncated values of $\lambda_a$.}
  \label{table1}
  \end{table}
  
We now turn to the distribution of divisors.  
Let $\mathcal{D}$ be the set defined in \eqref{Ddef}.
When $\theta(n)=tn$, where $t$ is constant,
Tenenbaum \cite[Lemma 2.2]{Ten86} showed that $\mathcal{D}=\mathcal{B}$.
We want to generalize this result as much as possible.
The example $\theta(n)=n+1$, for which $\mathcal{D}=\{1,2\}$ while $\mathcal{B}$ is infinite, 
illustrates that some condition on $\theta$ is required to ensure equality of these two sets. 
The condition we need is 
\begin{equation}\label{thetacond}
\theta(n) \le \theta(n+1), \quad m \theta(n) \le \theta(m n)\qquad (n,m\ge 1, \ \gcd(n,m)=1).
\end{equation}
Note that \eqref{thetacond} is satisfied if $\theta(n)/n$ is non-decreasing. 

\begin{theorem}\label{DB}
Let $\theta$ satisfy \eqref{theta}.
\begin{enumerate}
\item[(i)] We have $\mathcal{D}\subseteq \mathcal{B}$, hence $D(x)\le B(x)$. 
\item[(ii)] If $\theta$ satisfies \eqref{thetacond}, then $\mathcal{D}=\mathcal{B}$, hence $D(x)=B(x)$.
\end{enumerate}
\end{theorem}

As an example, consider $\theta(n)=n^2+1$. Theorems \ref{thm4} and \ref{DB} show 
that the number of integers $n\le x$ whose divisors 
$1=d_1< d_2 < \ldots <d_{\tau(n)}=n$ satisfy
$d_{j+1} \le d_j^2 +1$ for $1\le j < \tau(n)$,  is given by
$$D(x) = \frac{c x}{(\log x)^{\lambda_2}}\left(1+ O\left(\frac{1}{\log x}\right)\right),$$
where $c$ is a positive constant and $\lambda_2=0.4191...$

With $\theta(n)=2^n$, Corollary \ref{corB} and Theorem \ref{DB} show that the number of integers $n\le x$ whose divisors 
$1=d_1< d_2 < \ldots <d_{\tau(n)}=n$ satisfy
$d_{j+1} \le 2^{d_j}$ for $1\le j < \tau(n)$, is given by
$$D(x) = (1-L) x\left(1 + O\left(\frac{1}{\log x}\right)\right),$$
where $1-L=0.2265...$

The proof of Theorem \ref{thmB} is based on the functional equation in Lemma \ref{mainlemma}
and an estimate for the number of integers without small prime factors in Lemma \ref{philemma}.
Theorem  \ref{cor} is established with the help of Theorem \ref{thmB}.
The proof of Theorem \ref{thm4}, which requires the most amount of work,
is modeled after \cite[Theorem 1.2]{PDD}, with the added difficulty that 
the poles of the Laplace transform (i.e. the solutions of equation \eqref{lambdaeq}) 
now depend on the parameter $a$ (see Lemma \ref{lambdamu}).
The proof of Theorem \ref{DB} generalizes that of Tenenbaum \cite[Lemma 2.2]{Ten86}.

\section{Preliminaries}

Let
$$ \Phi(x,y) = \# \big\{ 1\le n\le x : P^-(n)>y \big\}  $$
and define $\Phi(x,\infty)=1$ for $x\ge 1$. 
For $u\ge 1$, Buchstab's function $\omega(u)$ is defined as the
unique continuous solution to the equation
\begin{equation*}
(u\omega(u))' = \omega(u-1) \qquad (u>2)
\end{equation*}
with initial condition
$u\omega(u)=1$ for $1\le u \le 2$.
Let $\omega(u)=0$ for $u<1$ and define $\omega$ at 1 and $\omega'$ at 1 and 2 by right-continuity. 
Let $\Gamma(u)$ denote the usual gamma function.

The calculation of the values of $\lambda_a$ in Table \ref{table1} and the 
approximation of several integrals in the proof of Lemma \ref{lambdamu} require
estimates for integrals involving $\omega(t)-e^{-\gamma}$.
For that purpose, we used exact formulas for $\omega(t)$ for $0\le t \le 5$, 
derived with the help of Mathematica. To estimate the contribution from $t>5$, 
we used a table of zeros and relative extrema of $\omega(t)-e^{-\gamma}$ on the interval $[5, 10.3355]$ due to Cheer and Goldston \cite{CG}, 
and the estimate (ii) from Lemma \ref{omega} for $t\ge 10.3355$.

\begin{lemma}\label{omega}
We have
\begin{enumerate}[(i)]
\item $|\omega'(u)|\le 1/\Gamma(u+1) \quad (u\ge 0)$,
\item $|\omega(u)-e^{-\gamma}| \le 1/\Gamma(u+1) \quad (u\ge 0)$.
\end{enumerate}
\end{lemma}

\begin{proof} See \cite[Lemma 2.1]{PDD}.
\end{proof}

\begin{lemma}\label{philemma}
Let $u=\log x / \log y$.
\begin{enumerate}
\item[(i)]
For $x\ge 1$, $y\ge 2$, we have
$$ \Phi(x,y)= e^\gamma x \omega(u) \prod_{p\le y} \left(1-\frac{1}{p}\right) +
O\left(\frac{y}{\log y} + \frac{x e^{-u/3}}{(\log y)^2}\right).$$
\item[(ii)]
 For $x\ge y \ge 2$ we have
$$ \Phi(x,y)=x \prod_{p\le y} \left(1-\frac{1}{p}\right) + O\left(\frac{x e^{-u/3}}{\log y}\right).$$
\end{enumerate}
\end{lemma}

\begin{proof}
(i) See \cite[Lemma 2.2]{PDD}.
Part (ii) follows from (i) and Lemma \ref{philemma}. 
\end{proof}

\begin{lemma}\label{mainlemma}
Let $\theta$ satisfy \eqref{theta}.
For $x\ge 0$ we have
\begin{equation*}
[x]=\sum_{n\le x} \chi(n) \, \Phi(x/n, \theta(n)) .
\end{equation*}
\end{lemma}

\begin{proof}
In \cite[Lemma 2.3]{PDD} we proved this result assuming $\theta(n)<\infty$.
The inclusion of $\infty$ in the possible values of $\theta$ has no effect on the proof. 
The basic idea is that every integer $1\le m\le x$ factors uniquely as $m=nr$, 
where $n \in \mathcal{B}$ and $r$ is counted in $\Phi(x/n,\theta(n))$. 
\end{proof}

\begin{lemma}\label{lemL}
Let $\theta$ satisfy \eqref{theta}.
The series 
\begin{equation*}
L=\sum_{n\ge 1} \frac{\chi(n)}{n} \prod_{p\le \theta(n)} \left(1-\frac{1}{p}\right)
\end{equation*}
converges and $0\le L \le 1$.
\end{lemma}

\begin{proof}
Lemma \ref{philemma} (ii) implies $\lim_{x\to \infty} \Phi(x,y)/x = \prod_{p\le y} (1-1/p)$.
From Lemma \ref{mainlemma}, we have
$$ \sum_{1\le n \le N} \frac{\chi(n)}{n} \prod_{p\le \theta(n)} \left(1-\frac{1}{p}\right) \le 1,$$
for every $N\ge 1$. The result now follows since the terms of the series are $\ge0$.
\end{proof}

\section{Proof of Theorems \ref{thmB} and \ref{cor}.}

\begin{proof}[Proof of Theorem \ref{thmB}]
Since $\Phi(x/n, \theta(n))=1$ when $n\le x < n\theta(n)$,
Lemma \ref{mainlemma} yields
\begin{equation*}
\begin{split}
[x] & = B(x) + \sum_{n\le x} \chi(n) \, \bigl(\Phi(x/n, \theta(n))-1\bigr) \\
& = B(x) + \sum_{n\theta(n) \le x}\chi(n) \, \bigl(\Phi(x/n, \theta(n))-1\bigr) \\
& = B(x) + x  \sum_{n\theta(n) \le x} \frac{\chi(n)}{n}  \prod_{p\le \theta(n)} \left(1-\frac{1}{p}\right)+ O(E_1(x)),
\end{split}
\end{equation*}
where
$$ E_1(x)= x \sum_{n\theta(n) \le x}\frac{\chi(n)}{n \log\theta(n)} \exp\left(-\frac{\log (x/n)}{3\log\theta(n)}\right),$$
by Lemma \ref{philemma} (ii).
Thus
$$[x]=B(x)+ Lx + O(E_1(x)+E_2(x)),$$
where
$$E_2(x)=  x\sum_{n\theta(n) > x}\frac{\chi(n)}{n \log\theta(n)}.$$
 This completes the proof of \eqref{thmB1}. The estimate \eqref{thmB2} follows, since $E_1(x)+E_2(x)=o(x)$ by
 Lemma \ref{lemL}.
\end{proof}

\begin{proof}[Proof of Theorem \ref{cor}]
(i) Assume $L<1$ so that  $ B(n)\ge c n $ for some $c>0$ and all $n \ge 1$. 
Let $g(n)=\log f(n)$ and assume $n g(n)$ is increasing. 
Partial summation applied twice yields
\begin{equation*}
\begin{split}
 \sum_{n\le N}  \frac{\chi(n)}{n g(n)} & = \frac{B(N)}{N g(N)} + \sum_{n\le N-1} B(n) \left(\frac{1}{n g(n)}-\frac{1}{(n+1)g(n+1)} \right) \\
& \ge  \frac{c N}{N g(N)} + \sum_{n\le N-1} c n  \left(\frac{1}{n g(n)}-\frac{1}{(n+1)g(n+1)} \right) \\
& = \sum_{n\le N}  \frac{c}{n g(n)},
 \end{split}
\end{equation*}
which grows unbounded as $N$ increases. But this is impossible, since
$$
\sum_{n\le N}  \frac{\chi(n)}{n g(n)}\le   \sum_{n\le N}  \frac{\chi(n)}{n \log \theta(n)}
\ll \sum_{n\le N}  \frac{\chi(n)}{n} \prod_{p\le \theta(n)} \left(1-\frac{1}{p}\right)
\le L \le 1,
$$
by Lemma \ref{lemL}. Thus $L=1$ and $B(x)=o(x)$.

(ii) Assume that $\theta(n) \ge f(n)\ge 2$, where $f(n) \gg P^+(n)$, and that
$$
\sum_{n\ge 1} \frac{1}{n \log f(n)}
$$
converges. 
Then there exists a $t\ge 1$ such that $t f(n) \ge P^+(n)$ and 
$$ L_{tf}\le \sum_{n\ge 1} \frac{1}{n} \prod_{p\le t f(n)} \left(1-\frac{1}{p}\right) < 1.$$
Theorem \ref{thmB} implies that $B_{tf}(x) \gg x$. 
According to the hypothesis, there is an $r\in \mathbb{N}$ such that $f(2^r n) \ge t f(n)$ for all $n\ge 1$.
Thus, if $n$ is counted in $B_{tf}(x/2^r)$, then $2^r n$ is counted in $B_f(x)$. This yields
$$B_\theta(x) \ge B_f(x) \ge B_{tf}(x/2^r) \gg_r x.$$

Part (iii) is an immediate consequence of Theorem \ref{thmB}.
\end{proof}

\section{Proof of Theorem \ref{thm4}.}

Throughout this section, we assume that $\theta$ satisfies \eqref{theta} and that
\begin{equation}\label{Best}
B(x) \ll \frac{x}{(\log x)^{\nu}} \qquad (x\ge 2),
\end{equation}
for some suitable $\nu \in [0,1]$, to be determined later. Clearly, $\nu=0$ is admissible.
All constants implied by $\ll$ and the big O notation may depend on $\theta$, and therefore on $a$, but are otherwise absolute.
Lemmas \ref{lem1} through \ref{lem4} correspond to Lemmas 5.3 through 5.7 of \cite{PDD}. The main difference is 
the assumption on the size of $B(x)$ for the purpose of estimating the error terms, for which we use \eqref{Best} here.  

\begin{lemma}\label{lem1}
For $x\ge e$ we have 
\begin{multline*}
B(x) = \\
x\ -\ x\sum_{n\theta(n)\le x} \frac{\chi(n)}{n} \, e^\gamma \omega\left(\frac{\log x/n}{\log \theta(n)}\right) 
\prod_{p\le \theta(n)} \left(1-\frac{1}{p}\right) + O\left(\frac{x}{(\log x)^{1+\nu}}\right)
\end{multline*}
\end{lemma}

\begin{proof}
As in the proof of Theorem \ref{thmB}, we have
\begin{equation}\label{lem0eq}
B(x)=[x]-\sum_{n\theta(n)\le x} \chi(n) \, (\Phi(x/n, \theta(n))-1).
\end{equation}
We apply Lemma \ref{philemma} (i) to estimate each occurrence of $\Phi(x/n, \theta(n))$ in \eqref{lem0eq}. 
The contribution from the error term $O(y/\log y)$ is
\begin{equation*}
\begin{split}
\ll \sum_{n\theta(n)\le x} \chi(n) \frac{\theta(n)}{\log\theta(n)}
 & \ll  \sum_{n^{1+a}\ll x} \chi(n)\frac{n^a}{\log 2n} \\
& \ll \frac{x^\frac{a}{1+a}}{\log x}\sum_{n\ll x^{1/(1+a)}} \chi(n) 
\ll \frac{x}{(\log x)^{\nu+1}},
\end{split}
\end{equation*}
by \eqref{Best}. 
For the contribution from the error term $O\left( \frac{xe^{-u/3}}{(\log y)^2}\right)$,
we split up the range of summation by powers of $2$ and use \eqref{Best} to get
\begin{equation*}
\begin{split}
 & \ll \sum_{n\theta(n) \le x} \chi(n) \frac{x}{n (\log \theta(n))^2} \exp\left(-\frac{\log x/n}{3\log \theta(n)}\right) \\
 & \ll \sum_{n\ll x^{1/(1+a)}} \frac{x}{n (\log 2n)^{2+\nu}} \exp\left(-\frac{\log 2x}{A\log 2n}\right) 
  \ll \frac{x}{(\log x)^{1+\nu}},
\end{split}
\end{equation*}
where $A$ is some suitable positive constant.
\end{proof}

\begin{lemma}\label{lem3}
For $x\ge e$ we have
\begin{equation*}
B(x)= x\sum_{n\ge 1} \frac{\chi(n)}{n \log \theta(n)} 
\left( e^{-\gamma} - \omega\left(\frac{\log x/n}{\log \theta(n)}\right)\right)
 + O\left(\frac{x }{(\log x)^{1+\nu}}\right)
\end{equation*}
\end{lemma}

\begin{proof}
We have $L=1$ by Theorem \ref{cor}.
Since $\omega(u)=0$ for $u<1$, Lemma \ref{lem1} shows that 
\begin{equation*}
B(x) = x\sum_{n\ge 1} \frac{\chi(n)}{n} 
\prod_{p\le \theta(n)} \left(1-\frac{1}{p}\right)
\left(1- e^\gamma \omega\left(\frac{\log x/n}{\log \theta(n)}\right)\right)
 + O\left(\frac{x }{(\log x)^{1+\nu}}\right)
\end{equation*}
The contribution from $n$ with $\log n \le \sqrt{\log x}$ is 
$\ll x \exp \left(-\sqrt{\log x}\right)$. 
For those $n$ for which $\log n > \sqrt{\log x}$, we use the estimate
$$ \prod_{p\le \theta(n)} \left(1-\frac{1}{p}\right) = \frac{e^{-\gamma}}{\log \theta(n)} 
\left(1+O\left(\frac{1}{(\log n)^4}\right)\right).$$
The contribution from the error term is 
$$\ll x \sum_{\log n > \sqrt{\log x}} \frac{1}{n (\log n)^5} \ll \frac{x}{(\log x)^2} \ll \frac{x}{(\log x)^{1+\nu}}.$$
\end{proof}

\begin{lemma}\label{lem3b}
For $x\ge e$ we have
\begin{equation*}
B(x)= x\sum_{n\ge 2} \frac{\chi(n)}{a n \log n} 
\left( e^{-\gamma} - \omega\left(\frac{\log x/n}{a \log n}\right)\right)
 + O\left(\frac{x }{(\log x)^{1+\nu}}\right).
\end{equation*}
\end{lemma}

\begin{proof}
Since $\theta(n) \asymp n^a$, we have $\log \theta(n) = a\log n + O(1)$. Inserting this estimate for 
each instance of $\log \theta(n) $ in Lemma \ref{lem3}, yields the desired result. 
For more details on the calculations see \cite[Lemma 5.6]{PDD}, where the case $a=1$ is treated.
\end{proof}

\begin{lemma}\label{lem4}
For $x\ge e$ we have
\begin{equation*}
B(x)= x\int_{e}^{\infty} \frac{B(y)}{a y^2 \log y} 
\left( e^{-\gamma} - \omega\left(\frac{\log x/y}{a\log y}\right)\right) \, \mathrm{d}y
 + O\left(\frac{x }{(\log x)^{1+\nu}}\right).
\end{equation*}
\end{lemma}

\begin{proof}
This follows from applying partial summation to the sum in Lemma \ref{lem3b}. 
\end{proof}

From Lemma \ref{lem4} we have, for $x\ge e$,
\begin{equation}\label{inteq}
B(x)= x \, \alpha - x \int_{e}^{\infty} \frac{B(y)}{ay^2 \log y} 
\, \omega\left(\frac{\log x/y}{a\log y}\right)\mathrm{d}y
 + O\left(\frac{x }{(\log x)^{1+\nu}}\right),
\end{equation}
where
$$ \alpha:=e^{-\gamma} \int_{e}^{\infty} \frac{B(y)}{a y^2 \log y} \, \mathrm{d}y . $$
For $x\ge e$ let $z\ge 0$ be given by  
$$z=\log \log(x)$$
and let
$$ G_\theta(z) := \frac{B\left(\exp(e^z)\right)}{\exp(e^z)}  = \frac{B(x)}{x} .$$
Dividing \eqref{inteq} by $x$ and changing variables in the integral
via $u=\log\log y$ we get, for $z\ge 0$,
\begin{equation}\label{conv}
\begin{split}
G_\theta(z) & =  \alpha - \frac{1}{a} \int_{0}^{z} G_\theta(u) \, \omega\left((e^{z-u} -1)/a\right) \, \mathrm{d}u +E_\theta(z) \\
       & = \alpha -\frac{1}{a}  \int_{0}^{z} G_\theta(u) \, \Omega_a(z-u)  \, \mathrm{d}u +E_\theta(z), 
\end{split} 
\end{equation}
where 
\begin{equation}\label{Error}
E_\theta(z) \ll  e^{-(1+\nu)z}
\end{equation}
and
$$ \Omega_a(u):= \omega\left((e^{u} -1)/a\right) . $$
Equation \eqref{conv} leads to the equation of Laplace transforms
$$ \widehat{G}_\theta(s) = \frac{\alpha}{s} -\frac{1}{a} \widehat{G}_\theta(s) \, \widehat{\Omega}_a(s) + \widehat{E}_\theta(s) \qquad (\re s >0),$$
which we solve for $\widehat{G}_\theta(s)$ to get
$$\widehat{G}_\theta(s) = \frac{\alpha}{s (1+\frac{1}{a}\widehat{\Omega}_a(s))} + \frac{\widehat{E}_\theta(s)}{1+\frac{1}{a}\widehat{\Omega}_a(s)}
\qquad (\re s >0).$$
 Let $F_a(z)$ be given by
 \begin{equation}\label{FaDef}
 F_a(z) =  1 - \frac{1}{a} \int_{0}^{z} F_a(u)  \, \Omega_a(z-u) \, \mathrm{d}u.
 \end{equation}
 Equation \eqref{FaDef} is an error-free, rescaled version of equation \eqref{conv},
 which depends on $a$, but does not involve $\theta$.
 Note that the upper limit of the integral could be replaced by $z-\log(a+1)$, since $\omega(t)=0$ for $t<1$.
 Thus $F_a(z)=1$ for $0\le z \le \log(a+1)$, and for $z>\log(a+1)$, $F_a(z)$ is determined by the values of 
 $F_a(u)$ for $u \in [0, z-\log(a+1)]$. Hence  \eqref{FaDef} defines the continuous function $F_a(z)$ for $z\ge 0$.
 As with $G_\theta(z)$, we find that the Laplace transform of $F_a(z)$ is given by
 $$  \widehat{F}_a(s) = \frac{1}{s (1+\frac{1}{a}\widehat{\Omega}_a(s))}  \qquad (\re s >0),$$
 and therefore,
 \begin{equation}\label{LaplaceGF}
     \widehat{G}_\theta(s) = \alpha \widehat{F}_a(s) + s \widehat{F}_a(s) \widehat{E}_\theta(s).
\end{equation}
The Laplace transform of $ \Omega_a(u)$, defined for $\re(s)>0$, is given by
\begin{equation*}
\begin{split}
 \widehat{\Omega}_a(s) & = \int_0^\infty \omega((e^u-1)/a) \, e^{-us} du = a \int_0^\infty \omega(t) \frac{dt}{(at+1)^{s+1}} \\
 &= \frac{e^{-\gamma}}{s(a+1)^s} + a\int_1^\infty \bigl( \omega(t) - e^{-\gamma}\bigr)   \frac{dt}{(at+1)^{s+1}} .
 \end{split}
 \end{equation*}
 The last equation extends $ \widehat{\Omega}_a(s)$ to a meromorphic function on $\mathbb{C}$ with a simple pole at $s=0$. 
 We will need to investigate the location of zeros of the entire function $g_a(s)$, defined by
 \begin{equation*}
\begin{split}
g_a(s) & = 1/  \widehat{F}_a(s) =s \left(1+\frac{1}{a}\widehat{\Omega}_a(s)\right) \\
& =s + \frac{e^{-\gamma}}{a (a+1)^s} 
+s \int_1^\infty \bigl( \omega(t) - e^{-\gamma}\bigr)   \frac{dt}{(at+1)^{s+1}}.
 \end{split}
 \end{equation*}

\pagebreak

\begin{lemma}\label{lambdamu}
For $a\ge 1$, $g_a(s)$ has a simple real zero at $s=-\lambda_a$, where $\lambda_1=1$, 
 $\lambda_a \in (0,1)$ for $a>1$,
$$
  \lambda_a>\frac{e^{-\gamma}}{a}\left(1+\frac{e^{-\gamma}\log (a+1)}{a}-\frac{0.16}{a}\right) =: l_a
$$
and
$$
\lambda_a< \frac{e^{-\gamma}}{a}\left(1+\frac{e^{-\gamma}\log (a+1)}{a}+\frac{\log^2(a+1)}{a^2}\right)=: u_a.
$$
For $a\ge 1$,  $g_a(s)$ has no other zero with  $\re(s)\ge -1-u_a$.
\end{lemma}

\begin{proof}
We have $g_a(0) = e^{-\gamma}/a >0$ and 
$$g_a(-1)=-1+e^{-\gamma}(1+1/a) - (2e^{-\gamma}-1)=e^{-\gamma}(1/a -1).$$ 
Thus $g_1(-1)=0$ and $\lambda_1=1$.
Assume $a>1$. We have $g_a(-1)<0$, so that $g_a(s)$ has a zero in the interval $(-1,0)$,
since $g_a(s)$ is real if $s$ is real.
For $s\in [-1,0]$ we have
\begin{equation*}
\begin{split}
 I_a(s) & :=\int_1^\infty \bigl( \omega(t) - e^{-\gamma}\bigr)   \frac{dt}{(at+1)^{s+1}} \\
 & \le   \frac{1}{(a+1)^{s+1}} \int_1^\infty | \omega(t) - e^{-\gamma}|  dt 
<  \frac{0.16}{(a+1)^{s+1}},
\end{split}
\end{equation*}
by numerical computation of the last integral, and
\begin{equation*}
\begin{split}
I_a(s)
& = \int_1^{e^\gamma} \bigl( t^{-1} - e^{-\gamma}\bigr)   \frac{dt}{(at+1)^{s+1}} 
+ \int_{e^\gamma}^\infty \bigl( \omega(t) - e^{-\gamma}\bigr)   \frac{dt}{(at+1)^{s+1}} \\
& \ge  \frac{1}{(ae^\gamma+1)^{s+1}} \bigl(\gamma - e^{-\gamma}(e^\gamma -1)\bigr)
-  \frac{1}{(ae^\gamma+1)^{s+1}}  \int_{e^\gamma}^\infty | \omega(t) - e^{-\gamma}|  dt \\
& \ge  \frac{\gamma -1 + e^{-\gamma} - 0.021}{(ae^\gamma+1)^{s+1}} 
> \frac{0.11}{(ae^\gamma+1)^{s+1}} .
\end{split}
\end{equation*}
Since $s\le 0$,
$$ g_a(s) <s + \frac{e^{-\gamma}}{a (a+1)^s} +s\frac{0.11}{(ae^\gamma+1)^{s+1}} =: g^+_a(s),$$
and
$$ g_a(s) >s + \frac{e^{-\gamma}}{a (a+1)^s}+s \frac{0.16}{(a+1)^{s+1}}=:g^-_a(s).  $$
Hence
$ g_a(-l_a)>g^-_a(-l_a)$ and $g_a(-u_a)<g^+_a(-u_a)$ if $u_a \le 1$.
A calculus exercise, made easier with the help of a computer, shows that $g^-_a(-l_a)>0$ and $g^+_a(-u_a)<0$ for $a\ge 1$. 
Hence $g_a(s)$ has a zero in $(-1,0)\cap (-u_a,-l_a)$.

It remains to show that, for $a\ge 1$,  $g_a(s)$ has no other zero with  $\re(s)\ge -1-u_a$.
We write
$$h_a(s)=g_a(s)-s, \qquad \mu_a=u_a+1.$$
For $\re(s)\ge -\mu_a$, 
$$|h_a(s)| \le  \frac{e^{-\gamma}}{a (a+1)^{-\mu_a}} 
+|s| \int_1^\infty \bigl| \omega(t) - e^{-\gamma}\bigr|   \frac{dt}{(at+1)^{-\mu_a+1}} .$$
Since $\mu_a>1$ and $at+1<(a+1)t$ for $t\ge 1$, the last integral is
$$ \le  (a+1)^{\mu_a-1} \int_1^\infty \bigl| \omega(t) - e^{-\gamma}\bigr| t^{\mu_a-1}  dt.$$
First assume $a\ge 10$. We have
$$\mu_a-1 \le  \mu_{10}-1<0.1,$$
$$   (a+1)^{\mu_a-1}\le  (10+1)^{\mu_{10}-1}< 1.2,$$
and
$$   \frac{e^{-\gamma}}{a (a+1)^{-\mu_a}} \le  \frac{e^{-\gamma}}{ 10(10+1)^{-\mu_{10}}} < 0.73.$$
Thus, for $a\ge 10$,
$$|h_a(s)| \le 0.73 + 1.2 |s|   \int_1^\infty \bigl| \omega(t) - e^{-\gamma}\bigr| t^{0.1}  dt<0.73+0.21|s|,$$
by numerical computation of the integral.
Hence we have $|h_a(s)|<|s|$ provided $|s|>0.93$.
On the boundary of the rectangle $R$ with corners $-\mu_a \pm iT$, $\alpha \pm iT$, where $\alpha, T \ge 5$, 
we have $|s|>0.93$, since $\mu_a>1$. 
Rouch\'e's theorem implies that $g_a(s)$ and $s$ have the same number of zeros, i.e. exactly one zero, with $\re(s)\ge -\mu_a$.

Next, assume $1\le a \le 10$. 
We have
$$\mu_a-1 \le  \mu_{1}-1<1.1,$$
$$   (a+1)^{\mu_a-1}\le  (1+1)^{\mu_{1}-1}< 2.1,$$
and
$$   \frac{e^{-\gamma}}{a (a+1)^{-\mu_a}} \le  \frac{e^{-\gamma}}{ (1+1)^{-\mu_{1}}} < 2.4.$$
Thus, for $1\le a\le 10$,
$$|h_a(s)| \le 2.4 + 2.1 |s|   \int_1^\infty \bigl| \omega(t) - e^{-\gamma}\bigr| t^{1.1}  dt<2.4+0.5|s|.$$
Hence $|h_a(s)|<|s|$ provided $|s|>4.8$. 
On the boundary of the rectangle $R$ we have $|s|\ge 5$, with the possible exception of the segment with endpoints $-\mu_a \pm 5i$.
For $s$ on this segment, and $1\le a \le 10$, we evaluate $|h_a(s)/s|$ numerically to get
$$\frac{|h_a(s)|}{|s|}=\frac{|h_a(-\mu_a+i\tau)|}{|-\mu_a+i\tau|} < 0.98, \quad (1 \le a \le 10, \ -5 \le \tau \le 5).$$
Hence $|h_a(s)|<|s|$ on the boundary of $R$, and 
Rouch\'e's theorem implies that $g_a(s)$ has exactly one zero with $\re(s)\ge  -\mu_a$.
\end{proof}

\pagebreak

\begin{lemma}\label{ga}
We write $s=\sigma + i \tau$ and
$$ H_a(\sigma):=  \frac{1}{a (a+1)^\sigma} 
+\frac{1}{a} \int_1^\infty  |\omega'(t)| \frac{dt}{(at+1)^{\sigma}} .$$
If $g_a(s)=0$ then $|\tau| \le H_a(\sigma)$. 
For $|\tau| \ge 2 H_a(\sigma)$, we have
$$ \frac{1}{g_a(s)}=\frac{1}{s} + O\left(\frac{H_a(\sigma)}{\sigma^2+\tau^2 } \right) .$$
\end{lemma}

\begin{proof}
Integration by parts shows that
$$
g_a(s) = s + \frac{1}{a (a+1)^s} 
+\frac{1}{a} \int_1^\infty  \omega'(t) \frac{dt}{(at+1)^{s}} 
 = s + H_a(\sigma) \xi_a(s),
$$
where $\xi_a(s) \in \mathbb{C}$ with  $|\xi_a(s)|\le 1$.
Thus any zero of $g_a(s)$ must satisfy $|\tau|\le |s| \le H_a(\sigma)$. 
We have
$$ g_a(s) = s \left(1+\frac{H_a(\sigma) \xi_a(s)}{s}\right),$$
from which the estimate for $1/g_a(s)$ follows.
\end{proof}

\begin{lemma}\label{FaInv}
The function $F_a(z)$ defined by \eqref{FaDef} satisfies
\begin{equation}\label{FaAsymp}
F_a(z)=c_a e^{-\lambda_a z} + O_a\left(e^{-\mu_a z}\right) \quad (z\ge 0),
\end{equation}
and
\begin{equation}\label{FapAsymp}
F'_a(z)=\tilde{c}_a e^{-\lambda_a z} + O_a\left(e^{-\mu_a z}\right) \quad (z\ge 0),
\end{equation}
for constants $c_a, \tilde{c}_a=-\lambda_a c_a$, where $\mu_a>\lambda_a+1$.
Here  $F'_a(z)$ denotes the right derivative of $F_a(z)$. 
\end{lemma}

\begin{proof}
We evaluate the inverse Laplace integral 
$$ F_a(z) = \frac{1}{2\pi i} \int_{1-i\infty}^{1+i\infty} \widehat{F}_a(s) e^{zs}\, \mathrm{d} s.$$
Let $\mu=\mu_a=u_a+1$ from Lemma \ref{lambdamu} and put $T=\exp(z(\mu+1))$.
Since the result is trivial for bounded $z$, we may assume that $z$ is sufficiently large such that $T>2H(-\mu)$. We have
$$ \int_{1+iT}^{1+i\infty}  \widehat{F}_a(s) e^{zs}\, \mathrm{d} s 
= \int_{1+iT}^{1+i\infty} \frac{1}{s} e^{zs}\, \mathrm{d} s + O\left(e^{z}/T \right)=O\left(e^{-\mu z}\right),$$
by Lemma \ref{ga} and integration by parts applied to the last integral.
We apply the residue theorem to the rectangle with vertices $-\mu \pm iT$, $1\pm iT$. 
The contribution from the horizontal segments can be estimated by Lemma \ref{ga} as 
$$ \int_{-\mu+iT}^{1+iT} \widehat{F}_a(s) e^{zs}\, \mathrm{d} s \ll \frac{e^{z}}{T}  =O\left(e^{-\mu z}\right).$$
For the vertical segment with $\re s = -\mu$ we have
$$ \left| \int_{-\mu-i2H_a(-\mu)}^{-\mu+i2H_a(-\mu)} \widehat{F}_a(s) e^{zs}\, \mathrm{d} s \right| \le 4H_a(-\mu) \max_{|\tau|\le 2H_a(-\mu)} \left|\widehat{F}_a(-\mu+i\tau)\right| e^{-\mu z} = O\left(e^{-\mu z}\right)$$
and
$$ \int_{-\mu+i2H_a(-\mu)}^{-\mu+iT} \widehat{F}_a(s) e^{zs}\, \mathrm{d} s = \int_{-\mu+i2H_a(-\mu)}^{-\mu+iT} \left(\frac{1}{s}+O(\tau^{-2})\right) e^{zs}\, \mathrm{d} s =O\left(e^{-\mu z}\right).$$
The residue theorem now yields
$$ F_a(z)= \res\left( \widehat{F}_a(s) e^{zs}; -\lambda_a \right)+O\left(e^{-\mu z}\right) = c_a e^{-\lambda_a z }+O\left(e^{-\mu z}\right) ,$$
which completes the proof of \eqref{FaAsymp}.

If $0\le z < \log(a+1)$, we have $F_a(z)=1$ and hence $F'_a(z)=0$. 
If $z\ge \log(a+1)$, \eqref{FaDef} implies
\begin{equation}\label{FapEq}
F'_a(z)= -\frac{1}{a} \int_0^{z-\log(a+1)} F_a(u) \, \omega'\left(\frac{e^{z-u}-1}{a}\right)  \frac{e^{z-u}}{a} \, \mathrm{d} u 
 -\frac{1}{a} F_a(z-\log(a+1)).
\end{equation}
The estimate \eqref{FapAsymp} follows from applying  \eqref{FaAsymp} to each occurrence of $F_a$ on the right-hand side of \eqref{FapEq},
and the fact that $\omega'(t) \ll_A t^{-A}$ for any constant $A$, by Lemma \ref{omega}.

\end{proof}

We are now ready to finish the proof of Theorem \ref{thm4}.
Since $\widehat{F'_a}(s)=s\widehat{F}_a(s)-F_a(0)$ and $F_a(0)=1$, \eqref{LaplaceGF} yields
$$    \widehat{G}_\theta(s) = \alpha \widehat{F}_a(s) + \widehat{E}_\theta(s) +\widehat{F'_a}(s) \widehat{E}_\theta(s),$$
and thus
\begin{equation}\label{last}
G_\theta(z)= \alpha F_a(z) + E_\theta(z) + \int_0^z F'_a(z-u) E_\theta(u) du.
\end{equation}
We estimate $F_a(z)$ and $F'_a(z-u)$ using Lemma \ref{FaInv}. 
Since $E_\theta(u) \ll e^{-(\nu+1)u}$ by \eqref{Error}, $\nu \ge 0$ and $\lambda_a\le 1$, 
 \eqref{last} shows that $G_\theta(z) \ll z e^{-\lambda_az}$. 
 Thus $\nu= \lambda_a-\varepsilon$ is acceptable in \eqref{Best} for every $\varepsilon>0$.
 With this choice of $\nu$,  \eqref{last} shows that $G_\theta(z) \ll  e^{-\lambda_az}$,
 which means that $\nu=\lambda_a$ is acceptable in \eqref{Best}.
Hence we assume $\nu = \lambda_a$ and $E_\theta(u) \ll e^{-(\lambda_a+1)u}$ for the remainder of this proof. 
The contribution to the integral in \eqref{last} from the main term in Lemma \ref{FaInv} is
\begin{equation*}
\begin{split}
& \int_0^z \tilde{c}_a  e^{-\lambda_a (z-u)} E_\theta(u) du \\
&=  \tilde{c}_a  e^{-\lambda_a z} \left(  \int_0^\infty e^{\lambda_a u} E_\theta(u) du 
+ O\left( \int_z^\infty e^{\lambda_a u} e^{-(\lambda_a+1)u} du\right) \right) \\
&=\tilde{c}_a   e^{-\lambda_a z}\left( \beta + O(e^{-z})\right),
\end{split}
\end{equation*}
say.
The contribution from the error term in Lemma \ref{FaInv} to the integral in \eqref{last} is
$$
\ll \int_0^z e^{-\mu_a(z-u)} e^{-(\lambda_a+1)u} du 
\ll e^{-(\lambda_a+1)z},
$$
since $\mu_a>\lambda_a +1$. Hence \eqref{last} implies
$$ G_\theta(z)= (\alpha c_a + \beta \tilde{c}_a )   e^{-\lambda_a z} + O\left(e^{-(\lambda_a+1)z}\right).$$
With 
$c_\theta :=\alpha c_a + \beta \tilde{c}_a$,
  we get
$$  G_\theta(z)=B(x)/x = c_\theta (\log x)^{-\lambda_a} + O \left((\log x)^{-(\lambda_a+1)}\right).$$
It remains to show that $c_\theta >0$ if $a\ge 1$. Since $\theta(n)\gg n^a \ge n$, there exists an integer $k$
such that $\theta(2^k n)\ge 2n$ for all $n\ge 1$. With $\theta_0(n)=2n$, we have $n\in \mathcal{B}_{\theta_0}$ implies 
$2^k n \in \mathcal{B}_\theta$. Hence $B_\theta(x) \ge B_{\theta_0}(x/2^k) \gg x/\log x$, by \cite[Theorem 1.2]{PDD}.
Since $\lambda_a+1>1$, we must have $c_\theta>0$.

\section{Proof of Theorem \ref{DB}}

(i) Assume that $n=p_1^{\alpha_1} \cdots p_k^{\alpha_k} \in \mathcal{D}$, where $p_1< p_2< \ldots < p_k$.
Let $0\le j < k$ and write $d_i=p_1^{\alpha_1} \cdots p_{j}^{\alpha_{j}}$. The next larger divisor, $d_{i+1}$, satisfies $d_{i+1}\ge p_{j+1}$, since $d_{i+1}$ must be divisible by at least one of the primes $p_{j+1},\ldots,p_k$.
Thus
$$ p_{j+1} \le d_{i+1} \le \theta(d_i) = \theta(p_1^{\alpha_1} \cdots p_{j}^{\alpha_{j}}),$$
which means that $n\in \mathcal{B}$.

(ii) Assume $\theta$ satisfies \eqref{thetacond}.
To show that $n\in \mathcal{B}$ implies $n\in \mathcal{D}$ for all $n\ge 2$,
we proceed by induction on $k$, the number of distinct prime factors of $n$. When $k=1$, $n=p^\alpha \in \mathcal{B}$ implies $p\le \theta(1)$.
For $1\le j \le \alpha$, we have
$$d_{j+1} = p^j = p^{j-1} p \le  p^{j-1} \theta(1) \le   \theta( p^{j-1}) =\theta(d_j),$$
which means that $n\in \mathcal{D}$.

Now assume that, for some $k\ge 1$ and each  $m=p_1^{\alpha_1} \cdots p_{k}^{\alpha_{k}}$, if $m \in \mathcal{B}$ 
then $m\in \mathcal{D}$. Let $n=m p^\alpha \in \mathcal{B}$, where $p>p_k$ and $m=p_1^{\alpha_1} \cdots p_{k}^{\alpha_{k}}$.
By the definition of $\mathcal{B}$, we have $m \in \mathcal{B}$, and hence $m\in \mathcal{D}$.
Assume that $1=t_1<t_2<\ldots < t_r=m$ are the divisors of $m$.
Then the divisors of $n$, say $1=d_1<d_2<\ldots <d_l=n$, are of the form
$ d_i = p^\beta t_j $, where $0\le \beta \le \alpha$, $1\le j \le r$. If $j<r$, we have
$$
d_{i+1} \le p^\beta t_{j+1} \le p^\beta \theta(t_j)  \le   \theta(p^\beta t_j)=\theta(d_i),
$$
as desired. If $j=r$, then $d_i=p^\beta m$ and we may assume $0\le \beta < \alpha$, that is $d_i<n$.
If $p>m$, then $p^{\beta+1}>p^\beta m=d_{i}$. Also, $n\in \mathcal{B}$ implies $p\le \theta(m)$, so
$$ d_{i+1}\le p^{\beta+1}\le p^\beta \theta(m) \le \theta( p^\beta m) = \theta(d_i).
$$
If $p \le m$, then $t_s \le m/p < t_{s+1}$, for some $s$ with $1\le s < r$. 
Now $d_i=p^\beta m<p^{\beta+1} t_{s+1} $, so
$$ d_{i+1} \le  p^{\beta+1} t_{s+1} \le  p^{\beta+1} \theta(t_s) \le \theta( p^{\beta+1} t_s).
$$
Since $t_s \le m/p$, we have $p^{\beta+1} t_s \le m p^\beta = d_i$ and hence $\theta( p^{\beta+1} t_s) \le \theta(d_i)$, as $\theta(n)$ is non-decreasing. Thus $d_{i+1} \le \theta(d_i)$ also holds in this case, which means that $n\in \mathcal{D}$.

\end{document}